\documentclass[11pt, letterpaper]{article}

\usepackage[
letterpaper,
top=1in,
bottom=1in,
left=1in,
right=1in]{geometry}

\usepackage[utf8]{inputenc}

\usepackage{amsthm, amsmath, amssymb}
\usepackage{mathtools}
\usepackage[dvipsnames]{xcolor}
\usepackage{dsfont}

\usepackage{enumitem}

\usepackage[colorlinks=true,urlcolor=blue,linkcolor=blue,citecolor=blue]{hyperref}
\usepackage[nameinlink]{cleveref}
\crefname{equation}{}{}

\usepackage{graphicx}

\usepackage{thmtools}
\usepackage{thm-restate}

\newtheorem{theorem}{Theorem}[section]
\newtheorem{corollary}[theorem]{Corollary}
\newtheorem{lemma}[theorem]{Lemma}

\theoremstyle{definition}

\newcommand{\E}{\mathbb{E}}

\newcommand{\R}{\mathbb{R}}
\newcommand{\Z}{\mathbb{Z}}

\newcommand{\G}{\mathcal{G}}

\newcommand{\SOSrk}{\mathrm{SOS}_\mathrm{rk}}
\newcommand{\MK}{\mathrm{MK}}

\newcommand{\sgn}{\mathrm{sgn}}
\newcommand{\ceil}[1]{\lceil #1 \rceil}
\newcommand{\floor}[1]{\lfloor #1 \rfloor}

\newcommand{\x}{\mathbf{x}}

\newcommand{\opt}{\mathrm{opt}}

\begin{document}

\title{On the Distribution of Unweighted Minimum Knapsack Instances with Large SOS Rank}

\author{
  \begin{minipage}[t]{0.3\textwidth}
    \centering
    Adam Kurpisz \\ \footnotesize \href{mailto:kurpisza@ethz.ch}{kurpisza@ethz.ch} \\ \small Bern Business School \& ETH Zurich
  \end{minipage}
  \hfill
  \begin{minipage}[t]{0.3\textwidth}
    \centering
    Lucas Slot\thanks{
    Supported by the project ``Foundations of sum-of-squares algorithms: worst-case and beyond'' with file number VI.Veni.242.234 of the research programme Veni ENW which is (partly) financed by the Dutch Research Council (NWO) under the grant \url{https://doi.org/10.61686/GIEOW65108}.        
    }  
     \\ \footnotesize 
     \href{mailto:l.f.h.slot@uva.nl}{l.f.h.slot@uva.nl} \\ \small University of Amsterdam
  \end{minipage}
  \hfill
  \begin{minipage}[t]{0.3\textwidth}
    \centering
    Mikhail Zaytsev \\ \footnotesize \href{mailto:mzaytsev@ethz.ch}{mzaytsev@ethz.ch} \\ \small ETH Zurich
  \end{minipage}
}

\maketitle

\begin{abstract}
We analyze the sum-of-squares rank of unweighted instances of the Minimum Knapsack (MK) problem, i.e., minimization of $\sum_{i=1}^n x_i$ for 0/1 variables under the constraint $\sum_{i=1}^n x_i \geq q$, with $q \in \R$. Such instances have long served as a testbed for understanding the limitations of lift-and-project methods in Boolean optimization. For example, both the Lovász--Schrijver and Sherali--Adams hierarchies require (maximal) rank~$n$ to solve them, already when $q=1/2$ is constant. The SOS hierarchy requires only \emph{sublinear} rank $O(\sqrt{n})$ to solve unweighted MK when $q=1/2$. On the other hand, when $q$ is allowed to vary with~$n$, the SOS rank of the problem may become linear. Interestingly, this is known to happen both when $q$ is large, and when $q$ is very small ($0<q \leq 2^{-n}$). This raises the question of whether we should think of hard instances of unweighted MK as being typical for the SOS hierarchy, or as a consequence of very specific choices of the threshold parameter $q$.

In this paper, we address this question by showing new upper and lower bounds on the SOS rank of unweighted MK in the whole regime of the parameter $q$. For $n-q \leq O(1)$, we show that the SOS rank is constant. In contrast, when $q \leq O(1)$, a linear rank is needed if $q$ is exponentially close to an integer. As our main positive result, we show that linear rank is very rare for $q \leq O(1)$. This can be expressed in the language of smoothed analysis: after perturbing $q$ by a Gaussian with mean $0$ and variance $\sigma^2$, the expected SOS rank of $\MK$ is $O(\sqrt{n} \log (n/\sigma))$.
\end{abstract}

\section{Introduction}
The sum-of-squares hierarchy is a systematic procedure to obtain tractable relaxations of (hard) optimization problems involving a polynomial objective function and polynomial constraints. It was introduced in the early 2000s by Lasserre~\cite{Lasserre01} and Parrilo~\cite{Parrilo2000}, based on earlier work on effective Positivstellens\"atze in real algebraic geometry~\cite{GrigorievV01, Putinar93, Schmuedgen:Positivstellensatz}. It has found extensive applications in mathematical optimization and computer science, see, e.g.,~\cite{Laurent2008, Lasserre2009, BarakSteurer2014}.

Here, we consider the important special case of linear optimization over the Boolean hypercube, i.e., problems of the form
\begin{equation}
\begin{split} \label{EQ:boolPOP}
      \opt := \min \quad& c^\top \x \\
      \mathrm{s.t.} \quad & A \x \geq b, \\
      &\x \in \{ 0, 1 \}^n,
\end{split} 
\end{equation}
Note that the Boolean constraints on $\x$ appearing in~\eqref{EQ:boolPOP}  can be enforced by degree-$2$ polynomial equalities $x_i = x_i^2$, for  $i \in [n]$. Formulation~\eqref{EQ:boolPOP} captures classical hard optimization problems, including independent set and max-cut, where the (first level) of the SOS hierarchy recovers the famous Lovász theta-number~\cite{Lovasz1979} and Goemans-Williamson relaxation~\cite{Goemans1995}, respectively.

Applied to this setting, the SOS hierarchy can be viewed as a \emph{lift-and-project method}, parameterized by the degree $d$, that yields tighter relaxations as the degree increases. The degree at which it reaches the convex hull of integer solutions is called rank, which for problem~\eqref{EQ:boolPOP} is at most $n$. 

\subsection{Unweighted Minimum Knapsack}
The Minimum Knapsack (MK) problem is given by
\[
\mathrm{MK}(c, w, q): \quad \min c^\top \x \quad \text{s.t.} \quad w^\top \x \ge q, \quad\x \in \{0,1\}^n,
\]
where each instance depends on a cost vector $c \in \mathbb{R}^n$, a weight vector $w \in \mathbb{R}^n$, and a threshold value $q \in \mathbb{R}$.
In this paper, we focus on symmetric instances of $\MK$ in which all weights and costs are equal to $1$, that is, instances of the form $\MK(\mathbf1, \mathbf1, q)$, or $\MK(q)$ for short.
Note that these are easy to solve by inspection: for $q < 0$, the optimum is $0$; for $q > n$, the problem is infeasible; and for $0 \leq q \leq n$, the optimum is $\lceil q \rceil$. 
Nonetheless, these `unweighted' knapsack instances have long served as a testbed for understanding the limitations of lift-and-project methods, including the sum-of-squares hierarchy. 
In early work, Cook and Dash~\cite{cookdash} showed that the Lovász--Schrijver hierarchy needs (maximal) rank $n$ to capture the polytope corresponding to the instance $\MK(1/2)$. Similarly, for the Sherali--Adams hierarchy, Laurent~\cite{Laurent03} proved that the rank is $n$. In the same work, Laurent observed that this polytope has SOS rank $2$ when $n=2$, but left the problem of determining the rank for general $n$ open.

More recently, Kurpisz, Potechin, and Wirth~\cite{kurpisz2021soscertificationsymmetricquadratic} show that $\MK(1/2)$ has \emph{sublinear} SOS rank $O(\sqrt{n})$ by proving an upper bound of $O(\sqrt{n} \log(1/q))$ for all $0 < q \leq 1/2$. This bound is tight for constant $q$ in light of a lower bound of $\Omega(\sqrt{n \log 1/q})$ established in~\cite[Lemma 14]{kurpisz:LIPIcs.ICALP.2019.79}.
When $q=q(n)$ is allowed to vary with $n$, the SOS rank of $\MK(q)$ may become linear: when $q\le 2^{-n}$, the lower bound of~\cite{kurpisz:LIPIcs.ICALP.2019.79} is $\Omega(n)$. Moreover, for a related variant of MK with an equality constraint of Grigoriev~\cite{Grigoryev}, the SOS rank is $\Omega(n)$ when both $q=\Theta(n)$ and $n-q=\Theta(n)$. Finally, $\MK(q)$ has SOS rank $0$ when $q$ is integral or $q\notin[0,n]$.

These results indicate that the SOS rank of unweighted MK depends sharply on the threshold parameter $q$, but only partial results are known in isolated regimes. This work aims to provide a more global picture, addressing in particular the following questions:
\begin{itemize}
    \item Do instances of $\MK(q)$ typically have sublinear SOS rank?
    \item Can we characterize when linear rank is necessary?
\end{itemize}

\subsection{Main Contributions}
We study the SOS rank of the unweighted minimum knapsack problem as a function of the threshold parameter~$q$.
Throughout the paper, it is convenient to write
\[
q = \lfloor q \rfloor + \widehat{q},
\]
where $\widehat{q} \in (0,1)$ denotes the fractional part of $q$.
Our results show that the SOS rank depends on~$q$ in two different ways: a baseline effect determined by which Hamming layers of the hypercube the constraint $|\x|\ge q$ intersects, and an additional effect that arises when the fractional part $\widehat{q}$ is very small.

\paragraph{Sources of hardness.}
In Section \ref{lb_section}, we establish lower bounds that clarify the roles of $\widehat{q}$ and~$\lfloor q \rfloor$.
Building on a result by Grigoriev~\cite{Grigoryev}, we show a baseline lower bound of
\[
\Omega\bigl(\min\{\lceil q\rceil,\,n-\lfloor q\rfloor\}\bigr)
\]
for all non-integral thresholds~$q$.
Beyond this baseline, we extend the linear lower bound on the SOS rank of $\MK(q)$ of~\cite{kurpisz:LIPIcs.ICALP.2019.79} when $0 < q \leq 2^{-n}$ to all $0<q<\floor{n/2}$ with fractional part in $(0, 2^{-n}]$, i.e., all $q$ that are \emph{slightly larger} than the nearest integer. 
This identifies an additional integrality-sensitive source of hardness. To be precise, we show that (see~\Cref{lb_MK}):

\begin{theorem}
    Let $0<q<\floor {n/2}$ with $\widehat{q}\ne 0$. The SOS rank of $\MK(q)$ is lower bounded by
    \[
    \begin{cases}
    \Omega\left(\min\left\{n, \sqrt{n \log 1/\widehat q}\right\}\right) & \text{ if } \quad 0 < \widehat q \leq 1/2,
    \\
    \Omega\left(\sqrt{n (1-\widehat{q})}\right) & \text{ if } \quad 1/2 < \widehat q < 1.
    \end{cases}
    \]
\end{theorem}

\paragraph{Upper bounds.}
In Section \ref{ub_section}, we complement these lower bounds with upper bounds on the SOS rank that are almost tight across the full range of the threshold parameter $q$. Previously, an upper bound of $O(\sqrt{n}\log 1/\widehat{q})$ was only known for $q \leq 1/2$; see~\cite{kurpisz2021soscertificationsymmetricquadratic}.
Notice that, while the lower bounds show that the SOS rank must necessarily depend on the fractional part $\widehat{q}$ when $0<q<\lfloor n/2\rfloor$, this dependence completely disappears once $q>\lfloor n/2\rfloor$.
We establish this in \Cref{upper_layers_ub} and \Cref{MK-rank} below, which can be summarized as:
\begin{theorem}
Let $0<q<n$ with $\widehat{q}\ne 0$.
The SOS rank of $\MK(q)$ is upper bounded by
\[
\begin{cases}
O\!\left(
\sqrt{n}\log\!\left(2/\widehat{q}\right)
+ \sqrt{n\lfloor q\rfloor}\log n
\right),
& \text{if } 0<q<\lfloor n/2\rfloor,\\[6pt]
O\!\left(n-\lfloor q\rfloor\right),
& \text{if } \lfloor n/2\rfloor<q<n.
\end{cases}
\]
\end{theorem}
As a consequence, we observe that sublinear SOS rank is typical for unweighted instances of the minimum knapsack problem when $q$ is sufficiently small. We make this intuition formal by considering `smoothed' choices of $q$, obtained by adding a small amount of (Gaussian) noise. Namely, we show that (see~\Cref{thm:smoothed-mk}):
\begin{theorem}
Let $0<q<\floor{n/2}$ with $\widehat{q}\ne 0$.
Let $\eta\sim\mathcal N(0,\sigma^2)$, where $\sigma=\sigma(n)\in(0,1)$ satisfies $\sigma=o(1)$.
Then, under perturbation of $q$ by $\eta$, the expected SOS rank of $\MK(q)$ is
\[
O\!\Bigl(
\sqrt{n}\bigl(
\sqrt{q}\,\log n
+ \log\!\bigl(1/\sigma\bigr)
\bigr)
\Bigr).
\]
\end{theorem}

In particular, for any $q=q(n)\le O(1)$, choosing $\sigma \ge 1/\mathrm{poly}(n)$ yields an expected SOS rank of $O(\sqrt{n}\log n)$ for unweighted instances of $\MK$ with perturbed threshold parameter $q+\mathcal N(0,\sigma^2)$. 

This result aligns with the framework of smoothed complexity, which interpolates between worst-case and average-case complexity.
The SOS method has been studied through this lens in several works~\cite{RandomSOSsphere, TCSsmooth1, TCSsmooth2, TCSsmooth3, TCSsmooth4}.
However, the smoothed analysis regimes treated there were closer in spirit to average-case models.

In contrast, we study smoothed complexity for the SOS method in a regime that is closer to the worst-case, where arbitrary instances are perturbed by small Gaussian noise, following the framework and motivation of the seminal works on smoothed analysis~\cite{SpielmanT04, DunaganST11}.

\subsection{Technical Overview}
To prove the results stated above, we combine extensions of known methods with new techniques, which we briefly overview here.

For the lower bounds, we use two techniques. 
In Section \ref{baseline_lb_section}, we make a simple transformation of optimality certificates from our setting to refutation certificates for MK with equality constraints and combine this with Grigoriev's lower bound \cite{Grigoryev}. 
In Section \ref{integrality_lb_section}, we follow the general strategy of \cite{kurpisz:LIPIcs.ICALP.2019.79}, but replace their reliance on approximation-theoretic results for the NOR function with Coppersmith–Rivlin type inequalities \cite{doi:10.1137/0523054, erdelyi2014coppersmithrivlintypeinequalitiesorder}. 
This allows us to extend the lower bound for the regime $0<q<\floor{n/2}$ with fractional part in $(0,1/2]$ to the full range $(0,1)$.

For the upper bounds, we treat the lower and upper Hamming layers of the hypercube separately. In Section \ref{upper_layers_ub_section}, we introduce a new approach based on Lagrange interpolation to construct SOS polynomials with the desired properties.
In particular, we provide a low-degree certificate for a polynomial with consecutive roots at integer points and one additional root placed arbitrarily; see Lemma \ref{SQF_remover}. This result enriches the toolbox of functions admitting low-degree SOS certificates and may be of independent interest.
For the lower layers, in Section \ref{lower_layers_ub_section}, we follow the construction of \cite{kurpisz2021soscertificationsymmetricquadratic} and extend its applicability from the regime $q\in(0,1/2]$ to all $q\in(0,n)$ with non-zero fractional part.
Finally, in Section \ref{smoothed_section}, we use analytical techniques to derive bounds on the SOS rank under small Gaussian perturbations of the threshold parameter.

The proofs rely on technical lemmas collected in Section \ref{technical_section}: some are new (Section \ref{technical_section_upper_layers}), while others adapt prior results with appropriate parameters (Section \ref{technical_section_lower_layers}).

\section{Preliminaries}
For $n\in\Z_{\ge 1}$, let $[n]:=\{1,\ldots, n\}$. 
Denote by $\R[\x]:=\R[x_1,\ldots, x_n]$ the ring of $n$-variate real polynomials. 
For any $p \in \R[\x]$, there is multilinear polynomial $\ell$ such that $p(\x) = \ell(\x)$ for all $\x \in \{0, 1\}^n$. We say $p$ is symmetric if $p(\x) = p(\pi \x)$ for all $\pi \in S_n$. For any such~$p$, there is a univariate polynomial $\tilde p$ (of the same degree) such that $p(\x) = \tilde p(|\x|)$ for all $\x \in \{0, 1\}^n$, where  $|\x|:=\sum_{i\in[n]}x_i$ denotes the Hamming weight. 

\paragraph{Sum-of-squares.}
Our results are formulated within the standard sum-of-squares (SOS) framework for Boolean polynomial optimization. In this framework, one studies problems of the form
\[
\mathcal P:\quad \min f(\x)\quad\text{s.t.}\quad g_i(\x)\ge0\ \text{ for } i\in[m],\quad \x\in\{0,1\}^n,
\]
where $f,g_1,\dots,g_m\in\R[\x]$. Boolean constraints are enforced via the polynomial equalities
$x_i^2-x_i=0$ for $i\in[n]$. Let $f^*$ denote the optimal value.
The degree-$d$ SOS program of $\mathcal P$ is
\[
f_\Sigma^d:=\max_{s_0,~s_i}\left\{\lambda\in\R: f-\lambda =  s_0 + \sum_{i\in[m]} s_i g_i \right\}.
\]
where $s_0$ and $s_i$ for $i \in [m]$ are SOS polynomials of degree $2d$ (up to a shift by $\deg(g_i)$). Note that $s_0 + \sum_{i\in[m]} s_i g_i$ is   nonnegative over the feasibility set of $\mathcal{P}$; thus $f^d_\Sigma \leq f^*$ for every $d \in \mathbb{Z}_{\ge 0}$.
The program is called exact if $f^d_{\Sigma}=f^*$. The smallest $d\in\Z_{\ge 0}$ such that the degree-$d$ SOS program is exact is called the SOS rank, which we  denote by
\[
\SOSrk(\mathcal P) := \min\left\{d\in \Z_{\ge 0}:f_{\Sigma}^{d}=f^*_{\G}\right\}.
\]
Over the Boolean hypercube, the degree-$n$ SOS program is exact, and a degree-$d$ program can be solved via a semi-definite program (SDP) of size $O(m\sum_{k=0}^d\binom{n}{k})$~\cite{Parrilo2000}.

Here, we follow the degree convention of the dual side of the SOS hierarchy, see Lasserre~\cite{Lasserre01}, which differs by an additive constant of one for the MK problem. This is the conventional choice in prior work on the MK problem. 
In the specific case of the unweighted minimum knapsack problem $\MK(q)$, a degree-$d$ SOS certificate of optimality consists of SOS polynomials $s_0,s_1$ satisfying
\[
|\x|-\lceil q\rceil = s_0(\x) + s_1(\x)(|\x|-q)
\qquad\text{for all }\x\in\{0,1\}^n,
\]
with $\deg(s_0)\le 2d+2$ and $\deg(s_1)\le 2d$. Note that any degree-bounded identity holding for all $\x \in \{0,1\}^n$ can be realized as a polynomial identity modulo the Boolean hypercube constraints, without increasing degree.

\section{Lower Bounds on the SOS Rank}\label{lb_section}

In this section, we establish lower bounds on the SOS rank of $\MK(q)$. Our results show that $\MK(q)$ has linear SOS rank for multiple ranges of the parameter $q$. Since the SOS-rank of $\MK(q)$ is $0$ when $q$ is an integer or lies outside $[0,n]$, we restrict our attention to $q\in(0,n)$ with fractional part~{ $\widehat{q}\ne 0$}. 

\subsection{A Baseline Source of Hardness}\label{baseline_lb_section}
The first lower bound is a simple consequence of Grigoriev's main result in \cite{Grigoryev}.

\begin{theorem} \label{PROP:GrigorievLB}
    Let $0<q<n$ with $\widehat{q}\ne 0$. The SOS rank of $\MK(q)$ is lower bounded by 
    \[
    \Omega(\min\{\ceil{q},n-\floor{q}\}).
    \]
\end{theorem}
\begin{proof}
    Suppose $0<q<\ceil{n/2}$ (the other case is symmetric). Assume there are SOS polynomials $s_0,s_1$ with
    \[
    |\x|-\ceil{q}=s_0(\x)+s_1(\x)(|\x|-q)\quad \text{for any $\x\in\{0,1\}^n$.}
    \]
    Then, we also have
    \[
    -1=h(\x)+g(\x)(|\x|-q)\quad \text{for any $\x\in\{0,1\}^n$,}
    \]
    where $g(\x)=\frac{s_1(\x)-1}{\ceil{q}-q}$ and $h(\x)=\frac{s_0(\x)}{\ceil{q}-q}$ is an SOS polynomial.

    However, by the main theorem in \cite{Grigoryev}, the degrees of $h$ and $g$ are lower bounded by $\Omega(\ceil{q})$, and the result follows.
\end{proof}

Thus, the SOS rank depends on the location of the constraint $|\x|\ge q$ relative to the Hamming layers. In particular, this shows the corollary below.

\begin{corollary}
    If $q=\Theta(n)$, $n-q=\Theta(n)$, and $\widehat{q}\ne 0$, then the SOS rank of $\MK(q)$ is $\Theta(n)$. 
\end{corollary}

When the constraint cuts through the middle Hamming layers, the SOS rank of $\MK(q)$ is necessarily linear.

\subsection{An Integrality Sensitive Source of Hardness}\label{integrality_lb_section}

The next theorem is based on the observation that when $\widehat{q}$ is small, i.e., $q$ is near an integer, in any SOS certificate $s_0,s_1$ for $\MK(q)$, the multiplier polynomial $s_1$ must take large values on Boolean points of Hamming weight $\floor{q}$, while remaining bounded on points of larger Hamming weight. By symmetry, see, e.g.~\cite{Grigoryev}, this reduces to the study of a univariate polynomial at the integer points $\floor{q},\ldots, n$ with sharply separated adjacent values. 

A degree lower bound for such polynomials, originally shown by Coppersmith and Rivlin in \cite{doi:10.1137/0523054} and used here in the form due to Erdélyi~\cite{erdelyi2014coppersmithrivlintypeinequalitiesorder}, yields the desired SOS rank lower bound.

\begin{theorem} \label{lb_MK}
    Let $0<q<\floor {n/2}$ with $\widehat{q}\ne 0$. The SOS rank of $\MK(q)$ is lower bounded by
    \[
    \begin{cases}
    \Omega\left(\min\left\{n, \sqrt{n \log 1/\widehat q}\right\}\right) & \text{ if } \quad 0 < \widehat q \leq 1/2,
    \\
    \Omega\left(\sqrt{n (1-\widehat{q})}\right) & \text{ if } \quad 1/2 < \widehat q < 1.
    \end{cases}
    \]
\end{theorem}

\begin{proof}
    Assume there exist SOS polynomials $s_0,s_1$ with
    \[
    |\x|-\ceil{q}=s_0(\x)+s_1(\x)(|\x|-q)\quad \text{for any $\x\in\{0,1\}^n$.}
    \]
    Since $s_0(\x)\ge 0$ for every $\x\in\{0,1\}^n$, this implies the following two points:
    \begin{enumerate}
        \item $s_1(\x)\ge \frac{|\x|-\ceil{q}}{|\x|-q}$ for any $\x\in\{0,1\}^n$ with $|\x|\le \floor{q}$,
        \item $0\le s_1(\x)\le \frac{|\x|-\ceil{q}}{|\x|-q}$ for any $\x\in\{0,1\}^n$ with $|\x|\ge \ceil{q}.$
    \end{enumerate}
    Consider the symmetrization of $s_1$, defined by
    \[
    S(\x):=\frac{1}{n!}\sum_{\pi\in S_n}s_1(\pi \x).
    \]
    Then $\deg(S)\le \deg(s_1)=:d$ and $S$ fulfills both points. 
    Recall that since $S$ is symmetric, there exists a univariate polynomial $\tilde{S}$ with $\deg(\tilde{S})\le d$ such that $S(\x)=\tilde{S}(|\x|)$.
    
    In particular, $\tilde{S}(\floor{q})\ge 1/\widehat{q}>0$ and $\tilde{S}(x)\le \frac{n-\ceil{q}}{n-q}\le 1$ for $x=\ceil{q},\ldots, n$.
    Thus, defining $P(x):=\frac{\tilde{S}(x+\floor{q})}{\tilde{S}(\floor{q})}$, we get 
    \[
    P(0)=1 \quad \text{and}\quad P(x)\in[0,\widehat{q}]\text{ for any } x\in[n-\floor{q}].
    \]
    By Theorem 2.2 in \cite{erdelyi2014coppersmithrivlintypeinequalitiesorder}, we obtain the bound
    \[
    \deg(P)=\Omega\left(\min\left\{n-\floor{q}, \sqrt{(n-\floor{q})\log\left(1/\widehat{q}\right)}\right\}\right),
    \]
    whenever $0<\widehat{q}\le 1/2$ and 
    \[
    \deg(P)=\Omega\left(\sqrt{(n-\floor{q})(1-\widehat q)}\right), 
    \]
    otherwise. Since $\deg(P)=\deg(\tilde{S})\le d$, the claim follows.
\end{proof}

The following corollary isolates the extreme regime in which $\widehat{q}$ is small, where the SOS rank becomes essentially maximal.

\begin{corollary}
    Let $0<q< \floor {n/2}$ with $\widehat{q}\ne 0$. If $\widehat{q}\in(0,2^{-n}]$, then the SOS rank of $\MK(q)$ is $\Theta(n)$.
\end{corollary}
Thus, when the constraint $|\x|\ge q$ cuts through the lower Hamming layers of the hypercube, SOS hardness becomes highly sensitive to integrality.

\section{Upper Bounds on the SOS Rank}\label{ub_section}

We establish upper bounds on the SOS rank in the upper and lower Hamming layers of the hypercube, and show that hardness in the lower layers is fragile.
In the next two subsections, we provide SOS certificates by defining
$s_i(\x):=\sigma_i(|\x|)$ for suitable univariate polynomials $\sigma_i$.
Throughout, $\sigma_0$ is defined in terms of $\sigma_1$ as
\[
\sigma_0(x):=x-\ceil{q}-(x-q)\sigma_1(x).
\]
We will repeatedly use the following lifting result, which allows us to convert
univariate nonnegativity into SOS certificates over the Boolean hypercube.

\begin{lemma} \label{lifting_to_hypercube}
Let $p$ be a univariate polynomial that is nonnegative on $[0,n]$.
Then $p(|\x|)$ admits an SOS certificate over $\{0,1\}^n$ of degree $O(\deg(p))$.
More precisely, depending on the parity of $\deg(p)$, there exist SOS polynomials
$t_1,t_2$ such that
\[
p(x)=t_1(x)+t_2(x)\,x(n-x)
\quad\text{or}\quad
p(x)=x\,t_1(x)+(n-x)\,t_2(x),
\]
and both $|\x|$ and $n-|\x|$ admit SOS certificates of constant degree over $\{0,1\}^n$.
\end{lemma}
This representation is a consequence of classical results on nonnegative univariate
polynomials on an interval; we use it here
in a form convenient for our setting. See, e.g.,~\cite[Theorem~3.72]{doi:10.1137/1.9781611972290}.

\subsection{The Upper Layers of the Hypercube}\label{upper_layers_ub_section}

We begin with the upper Hamming layers of the hypercube, corresponding to parameters $\lfloor n/2\rfloor < q < n$. 
Our goal is to construct a univariate polynomial $\sigma_1$ such that the resulting polynomials $s_0$ and $s_1$ have SOS certificates of small degree.
In our construction, we exploit the following lemma, inspired by~ \cite[Section 4]{10.5555/2982445.2982462}.

\begin{lemma} \label{SQF_remover}
    Let $k\in[n]$ and $r\in(k-1,k]$. The polynomial
    \[
    A_{k,r}(\x)=(|\x|-r)\prod_{j=0}^{k-1}(|\x|-j)
    \]
    has an SOS certificate of degree $O(k)$.
\end{lemma}
\begin{proof}
    For $j=0,\ldots, n$, let $e_j(\x)=\binom{|\x|}{j}$. Using basic binomial identities, it can be readily checked that
    \begin{align*}
    A_{k,r}(\x)
    = (k-r)k!e_k(x_1^2,\dots,x_n^2) + (k+1)!e_{k+1}(x_1^2,\dots,x_n^2),
    \end{align*}
    for all $\x\in\{0,1\}^n$, which is an SOS polynomial of degree $2k+2$.
\end{proof}

The idea is to force the polynomial $\sigma_0$ to have roots exactly in $\ceil{q}, \ceil{q}+1,\ldots, n$ and a single other point $\floor{q}\le r<\ceil{q}$, and to apply Lemma \ref{lifting_to_hypercube} in combination with Lemma \ref{SQF_remover} to obtain the desired certificate.
Thus, we need $\sigma_1$ to cross the red hyperbola in Figure \ref{corrected_interpolation} in exactly the points~{$r, \ceil{q}, \ldots, n$}.

\begin{figure}
    \centering
    \includegraphics[width=0.8\linewidth]{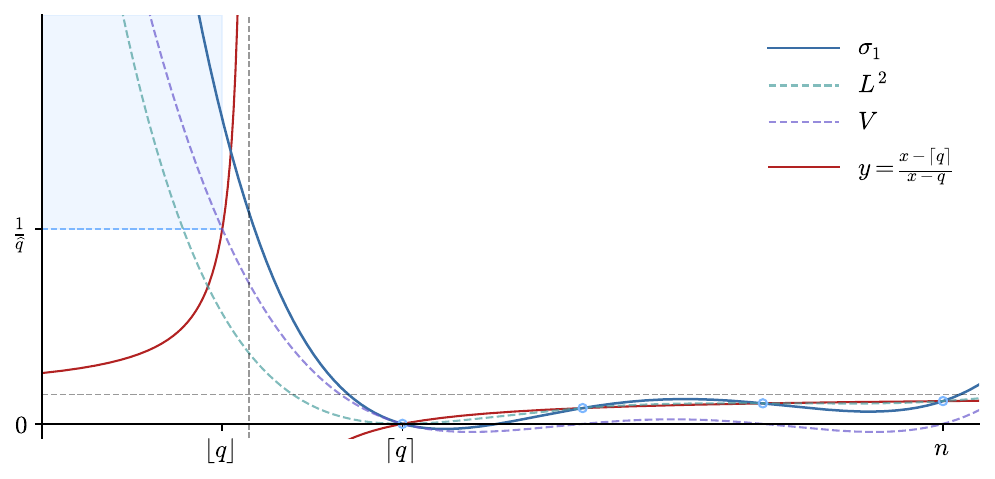}
    \caption{Interpolation geometry for the univariate $\sigma_1$ when $q>\floor{n/2}$.}
\label{corrected_interpolation}
\end{figure}

We therefore start with an SOS polynomial $L^2$, where $L$ interpolates a function whose square matches the hyperbola at the points $\ceil{q},\ceil{q}+1,\ldots, n$.
However, when the pole at $q$ is too close to $\floor{q}$, the function $L^2$ may exit the blue region in Figure \ref{corrected_interpolation} too early, making $\sigma_0$ negative in $\floor{q}$. Therefore, we compensate by adding a polynomial $V$ that is large enough in $\floor{q}$, vanishes in $\ceil{q},\ldots, n$, and admits an SOS certificate of small degree.

\begin{theorem}\label{upper_layers_ub}
    Let $\lfloor n/2\rfloor < q < n$ with $\widehat{q}\ne 0$. 
    Then the SOS rank of $\mathrm{MK}(q)$ is upper bounded by $$O\left(n-\floor{q}\right).$$
\end{theorem}
\begin{proof}
    We proceed in two steps.
    \paragraph{Constructing $\sigma_1$.} 
    We impose the following constraints on $\sigma_1$
    \begin{enumerate}
        \item $\sigma_1(x)=\frac{x-\ceil{q}}{x-q}$ for all $x=\ceil{q},\ldots, n$,
        \item $\sigma_1(x)\ge\frac{1}{\widehat{q}}$ for all $0\le x\le \floor{q}$.
    \end{enumerate}
    Consider 
    \[
    f:\R\setminus[q, \ceil{q})\rightarrow\R, \ x\mapsto\sqrt{\frac{x-\ceil{q}}{x-q}}
    \]
    and the interpolating polynomial $L$ of degree at most $d=n-\ceil{q}$ for the $d+1=n-\floor{q}$ points
    \[
    (x_j, y_j)=\left(j, f(j)\right)\text{ for all } j=\ceil{q},\ldots, n,
    \]
    let $V(x):=\frac{1}{\widehat{q}(n-\floor{q})!}\prod_{j=\ceil{q}}^n(j-x)$, and define
    \[
    \sigma_1(x):=L(x)^2+V(x).
    \]
    Notice that both constraints are satisfied by construction of $\sigma_1$. This is clear, since on the one hand $V(x)=0$ for all $x=\ceil{q},\ldots, n$, and on the other hand, $V(\floor{q})=\frac{1}{\widehat{q}}$ and $V$ is decreasing on~${(-\infty, \floor{q}]}$, so
    \[
    L(x)^2+V(x)\ge V(x)\ge V(\floor{q})= \frac{1}{\widehat{q}},
    \]
    for all $0\le x \le \floor{q}$.
    Additionally, 
    \[
    s_1(\x)=L(|\x|)^2+\frac{1}{\widehat{q}(n-\floor{q})!}A_{n-\ceil{q},n-\ceil{q}}(\mathbf1-\x)
    \]
    has an SOS certificate of degree $O\left(n-\floor{q}\right)$, where we used that $A_{n-\ceil{q},n-\ceil{q}}$ has an SOS certificate of degree $O(n-\floor{q})$ by Lemma \ref{SQF_remover}, and $L(|\x|)^2$ is already a square of degree at most $2(n-\ceil{q})$.
    
    \paragraph{Bounding the SOS degree of $s_0$.}
    
    By Lemma \ref{roots_of_sigma0_interpolation}, the roots of $\sigma_0$ are exactly $\ceil{q},\ldots, n$ together with a single root $\floor{q}\le r<\ceil{q}$, all of odd multiplicity. This allows us to write
    \begin{align*}
    &\sigma_0(x)=p(x)(x-r)\prod_{j=\ceil{q}}^n(x-j)^{2\mu_j+1}\\
    &=(-1)^{n-\ceil{q}}p(x)(n-x-(n-r))\prod_{j=0}^{n-\ceil{q}}(n-x-j)\prod_{j=\ceil{q}}^n(x-j)^{2\mu_j}
    \end{align*}
    where $2\mu_j+1$ are the multiplicities of each root $j=\ceil{q},\ldots, n$ and $p$ is a polynomial of degree $O(\deg(\sigma_1))$ that has no additional zeros in $[0,n]$.
    Importantly, note that $(-1)^{n-\ceil{q}}p(x)$ is positive on $[0,n]$ since $\sigma_0(x)>0$ for all $0\le x<\floor{q}$. Indeed,
    \begin{align*}
    1&=\mathrm{sgn}(\sigma_0(x))=\mathrm{sgn}(p(x))\mathrm{sgn}\left((x-r)\prod_{j=\ceil{q}}^n(x-j)\right)\\
    &=\mathrm{sgn}(p(x))(-1)^{n-\ceil{q}}
    \end{align*}
    and $p$ has no roots in $[0,n]$, so it cannot change sign.
    We can thus apply Lemma \ref{lifting_to_hypercube} to conclude that $(-1)^{n-\floor{q}}p$ has an SOS certificate of degree $O(\deg(\sigma_1))$. 
    
    The product $\prod_{j=\ceil{q}}^n(x-j)^{2\mu_j}$ is already SOS of degree at most $O(\deg(\sigma_1))$. Additionally, we observe that
    \[
    s_0(\x)=(-1)^{n-\ceil{q}}p(|\x|)A_{n-\floor{q}, n-r}(\mathbf1-\x)\left(\prod_{j=\ceil{q}}^n(|\x|-j)^{\mu_j}\right)^2.
    \]
    By another application of Lemma \ref{SQF_remover}, the polynomial $A_{n-\floor{q}, n-r}(\mathbf1-\x)$ has an SOS certificate of degree $O(n-\floor{q})$ over the Boolean hypercube.
    Therefore, both $s_0$ and $s_1$ have SOS certificates of degree $O(n-\floor{q})$, which concludes the proof.
\end{proof}

\subsection{The Lower Layers of the Hypercube}\label{lower_layers_ub_section}

We now investigate the lower Hamming layers of the hypercube, corresponding to parameters $0<q<\lfloor n/2\rfloor$. 
We construct an SOS polynomial $\sigma_1$ of small degree that remains in the blue regions illustrated in Figure \ref{fig_Cheb}, i.e., that is relatively large when $0\le x\le \floor{q}$ and stays very small when $\ceil{q}\le x \le n$. 
    
\begin{figure}[t]
  \centering
  \includegraphics[width=0.8\linewidth]{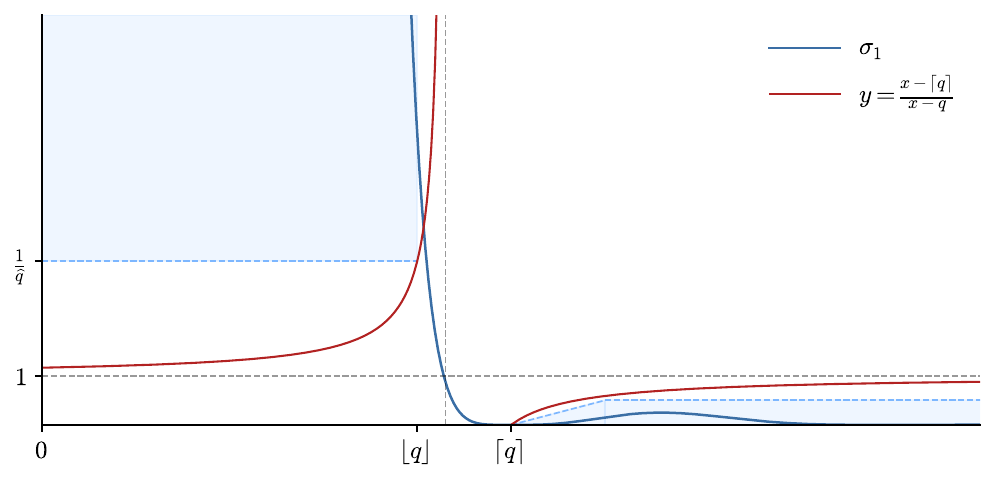}
  \caption{Geometric constraints on the univariate $\sigma_1$ when $q<\floor{n/2}$.}
  \label{fig_Cheb}
\end{figure}

Our construction generalizes the one in \cite{kurpisz2021soscertificationsymmetricquadratic} based on Chebyshev polynomials of the first kind, known for their extremal growth properties. We shift and scale a Chebyshev polynomial of appropriate degree so that its first root is at $\ceil{q}$. The polynomial we obtain will decay very quickly to the left of $\ceil{q}$ and, therefore, crosses the hyperbola somewhere between $\floor{q}$ and $\ceil{q}$.

We will thus obtain that $\sigma_0$ is nonnegative in $[0,n]\cap\Z$ and, by continuity, has only one root between $\floor{q}$ and $\ceil{q}$. We will then call upon a result regarding symmetric quadratic functions, shown by Kurpisz et al. in \cite{kurpisz2021soscertificationsymmetricquadratic}, and lift everything to obtain an SOS certificate for $s_0$ with Lemma \ref{lifting_to_hypercube}.

\begin{theorem}\label{MK-rank}
    Let $0<q<\floor {\frac{n}{2}}$ with $\widehat{q}\ne 0$. 
    Then the SOS rank of $\mathrm{MK}(q)$ is upper bounded by
    \[
    O\!\left(\sqrt{n}\log\!\left(2/\widehat{q}\right) + \sqrt{n\floor{q}}\log n\right).
    \]
\end{theorem}
\begin{proof}
    We proceed in two steps.
    \paragraph{Constructing $\sigma_1$.}
    We impose the following constraints on $\sigma_1$
    \begin{enumerate}
        \item $\sigma_1(x)>\frac{1}{\widehat{q}}$ for all $0\le x\le \floor{q}$,
        \item $\sigma_1(x)\le \frac{x-\ceil{q}}{2}$ for all $\ceil{q}\le x \le \ceil{q}+1$,
        \item $\sigma_1(x)\le \frac{1}{2}$ for all $\ceil{q}+1\le x\le n$.
    \end{enumerate}
    
    Let $T_d$ be the Chebyshev polynomial of the first kind. Denote by $r_0$ the smallest root of $T_d\left(\frac{x}{n-\floor{q}}-1\right)$ and define
    \[
    \tau(x):=\frac{1}{8}T_d\left(\frac{x-\ceil{q}+r_0}{n-\floor{q}}-1\right)^m \text{ and }
    \sigma_1(x):=\tau(x)^2
    \]
    for integers $d=\lceil 3\sqrt{n-\floor{q}}\rceil$ and $m=\lceil\frac{1}{2}\log_2(\frac{64}{\widehat{q}})\rceil$. Then $\sigma_1$ satisfies conditions \emph{(1), (2),} and \emph{(3)} by Lemma \ref{choosing constants}.
    We therefore obtain $\deg(\sigma_1)= O(\sqrt{n-\floor{q}}\log(2/\widehat{q}))$. Additionally, $\sigma_1$ is already an SOS polynomial since it is a square.
    
    \paragraph{Bounding the SOS Degree of $s_0$.}

    From the previous observations, $\sigma_0(x)>0$ for all $x\in[0,\floor{q}]\cup(\ceil{q},n]$ with $\sigma_0(\ceil{q})=0$. By Lemma \ref{roots_of_sigma0_Chebyshev}, $\sigma_0$ has exactly two roots, namely $\ceil{q}$ and $r\in(\floor{q},q)$, both of which are simple. Thus, we may write 
    \[
    \sigma_0(x)=(x-\ceil{q})(x-r)p(x)
    \]
    for a polynomial $p$ with $\deg(p)= O(\deg(\sigma_1))$ positive on $[0,n]$. 
    By Lemma \ref{lifting_to_hypercube}, $p(|\x|)$ has an SOS certificate of degree $O(\deg(p))$. We may use two results from \cite{kurpisz2021soscertificationsymmetricquadratic}, Theorem 1 on symmetric quadratic functions and Corollary 9, to conclude that the polynomial $(|\x|-\ceil{q})(|\x|-r)$ admits an SOS certificate of degree at most $O(\sqrt{n\ceil{q}}\log(n))$ for $2\le \ceil{q}\le \lfloor n/2\rfloor$ and $O(1)$ for $\ceil{q}=1$. 
    Thus, the polynomial $s_0(\x)=\sigma_0(|\x|)$ admits a certificate of degree at most $O\left(\deg(\sigma_1) +\sqrt{n\floor{q}}\log(n)\right)$. This establishes the final bound.
\end{proof}

\subsection{Limits of Integrality Sensitive Hardness}\label{smoothed_section}

At a high level, the pointwise bounds derived above show that the SOS rank of $\MK(q)$ depends on the fractional part of $q$ when $q<\floor {n/2}$. 
In particular, the rank is $\Theta(n)$ only when $q$ lies in a very narrow region close to the lower endpoint $\floor{q}$, and remains comparatively small elsewhere. 
We quantify this effect by expressing an upper bound for the SOS rank under a perturbation of $q$ by $\sigma=\sigma(n)=o(1)$.

\begin{theorem}\label{thm:smoothed-mk}
Let $0<q<\floor{n/2}$ with $\widehat{q}\ne 0$.
Let $\eta\sim\mathcal N(0,\sigma^2)$, where $\sigma=\sigma(n)\in(0,1)$ satisfies $\sigma=o(1)$.
Then, under perturbation of $q$ by $\eta$,
\[
\E_{\eta} \!\left[\SOSrk\!\bigl(\MK( \eta+q)\bigr)\right]
= O\!\Bigl(
\sqrt{n}\bigl(
\sqrt{q}\,\log n
+ \log\!\bigl(1/\sigma\bigr)
\bigr)
\Bigr).
\]
\end{theorem}
\begin{proof}
    The proof is a sequence of routine estimates. We use Theorem \ref{MK-rank} to compute
    \begin{align*}
      \E_{\eta} &\!\left[\SOSrk\!\bigl(\MK( \eta+q)\bigr)\right]\\
      & = \frac{1}{\sigma\sqrt{2\pi}}\int_{\R}
         \exp\!\left(-\frac{x^2}{2\sigma^2}\right)
         \SOSrk\!\bigl(\MK(x+q)\bigr)\,dx\\
      & = \frac{1}{\sigma\sqrt{2\pi}}
         \sum_{k\in[n]}\int_{k-1-q}^{k-q}
         \exp\!\left(-\frac{x^2}{2\sigma^2}\right)
         \SOSrk\!\bigl(\MK(x+q)\bigr)\,dx\\
      & = O\!\left(\frac{\sqrt{n}}{\sigma}
         \Biggl[
           \sum_{k\in[n]}
             \bigl(I_k + J_k \sqrt{k}\log n\bigr)
         \Biggr]\right),
    \end{align*}
    where, for every $k\in[n]$, we define
    \begin{align*}
        I_k &:= \int_{k-1-q}^{k-q}\exp\left(-\frac{x^2}{2\sigma^2}\right)\log\left(\frac{1}{x+q-k+1}\right)dx, \\
        J_k &:= \int_{k-1-q}^{k-q}\exp\left(-\frac{x^2}{2\sigma^2}\right)dx.
    \end{align*}
    Let us consider two cases: $k\in C:=\{\floor{q},\ceil{q},\ceil{q}+1\}$ and $k\in [n]\setminus C$.
    \paragraph{Bounding $I_j$.}

    Set $u=q+x+1-k$, then
    \begin{align*}
        I_k = \int_0^1\phi_{q-k+1, \sigma}(u)\log\left(\frac{1}{u}\right)du,
    \end{align*}
    where $\phi_{q-k+1, \sigma}(u)=\exp\left(-\frac{(u-q+k-1)^2}{2\sigma^2}\right)$. \\
    For $k\in C$, we find
        \begin{align*}
            I_k&=\int_0^\sigma\phi_{q-k+1, \sigma}(u)\log\left(\frac{1}{u}\right)du+\int_\sigma^1\phi_{q-k+1, \sigma}(u)\log\left(\frac{1}{u}\right)du\\
            & \le \int_0^\sigma\log\left(\frac{1}{u}\right)du+\log\left(\frac{1}{\sigma}\right)\int_\sigma^1\phi_{q-k+1, \sigma}(u)du\\
            & \le \sigma\left(1+\log\left(\frac{1}{\sigma}\right)\right)+\sigma\log\left(\frac{1}{\sigma}\right)\sqrt{2\pi}\\
            &=O\left(\sigma\left(1+\log\left(\frac{1}{\sigma}\right)\right)\right),
        \end{align*}
        where we used standard bounds in each step.
        
    For the second case, $k\in [n]\setminus C$, we have $|\ceil{q}-k|\ge 2$. We use the following bound
    \begin{align*}
        I_k &= \int_0^1\phi_{q-k+1,\sigma}(u)\log\left(\frac{1}{u}\right)du\\
        &\le \exp\left(-\frac{(|\ceil{q}-k|-1)^2}{2\sigma^2}\right)\int_0^1\log\left(\frac{1}{u}\right)du \le \exp\left(-\frac{1}{2\sigma^2}\right)
    \end{align*}
    which holds since $u\in[0,1],$ so $u+\ceil{q}-1-q\in [\ceil{q}-1-q, \ceil{q}-q]\subseteq[-1,1]$, and $|u+k-1-q|\ge|k-\ceil{q}|-|u+\ceil{q}-1-q|\ge |k-\ceil{q}|-1$.
    
    \paragraph{Bounding $J_j$.}
    We treat two cases. First, assume $|k-\ceil{q}|\ge 2$ and let $u=x+q-k+1$. Just as before, we find
    \begin{align*}
        J_k &= \int_{0}^1\phi_{q-k+1,\sigma}(u)du\\
        &\le \exp\left(-\frac{(|k-\ceil{q}|-1)^2}{2\sigma^2}\right) \le \exp\left(-\frac{1}{2\sigma^2}\right).
    \end{align*}
    For $k\in C$, we use the following bound instead
    \[
    J_k\le\int_\R\exp\left(-\frac{x^2}{2\sigma^2}\right)dx=\sigma\sqrt{2\pi}.
    \]

    \paragraph{Combining the bounds.}
    The only non-negligible contributions come from $I_k$ and $J_k$ when $k\in C$. Plugging these into the formula we found in the beginning yields exactly the desired bound.
\end{proof}
Now, for $q$ sufficiently small and $1/\sigma$ polynomial in $n$, the bound becomes $O(\sqrt{nq}\log(n))$, which is $o(n)$.

\section{Technical Lemmas}\label{technical_section}
    In the proofs of Theorems \ref{upper_layers_ub} and \ref{MK-rank}, we left some technical details unproven. We complete the aforementioned proofs in this section.
\subsection{Technical Lemmas for Section \ref{upper_layers_ub_section}}\label{technical_section_upper_layers}
    The only technical result we left open in Section \ref{upper_layers_ub_section} was Lemma \ref{roots_of_sigma0_interpolation}. To show it, we will need a couple of preliminaries.

    In the special case where $\ceil{q}=n$, $L=0$, and everything works out as intended, which can be easily checked. Thus, we restrict our attention to the case $\ceil{q}<n$.
    First, note that ${L(\ceil{q})=0}$. Thus, we may factor $L(x)=(x-\ceil{q})M(x),$ where $M$ is the minimum degree interpolating polynomial for the function
        \[
        h:(\ceil{q},n]\rightarrow\R, x \mapsto \frac{1}{\sqrt{(x-\ceil{q})(x-q)}}
        \]
    at the nodes $\ceil{q}+1,\ \ldots,\ n$. Proving things about $M$ will be more convenient since the function~$h$ is well behaved in $[\ceil{q}+1,n]$, as opposed to $f$, which is not differentiable in $x=\ceil{q}$.

    Since $L$ is interpolating the square root of the desired function, it will be beneficial to show that the resulting polynomial remains positive on $(\ceil{q},n]$ and non-positive on $[0,\ceil{q})$. This way, we can safely square everything to obtain the desired statements.
    \begin{lemma} \label{L_on_first_segment}
        The function $L:\R\rightarrow\R$ is non-positive, increasing on $(-\infty, \ceil{q})$, and positive on~$(\ceil{q},n]$.
    \end{lemma}
    \begin{proof}
        We show that $M$ is non-increasing on $(-\infty, \ceil{q}+1)$. Then
        \[
        L'(x)=(x-\ceil{q})M'(x)+M(x)\ge M(\ceil{q}+1)=\frac{1}{\sqrt{\ceil{q}+1-q}} >0
        \]
        for any $x\in(-\infty, \ceil{q})$, since $(x-\ceil{q})M'(x)\ge 0$ and $M(x)\ge M(\ceil{q}+1)$. In particular, $L(x)\le L(\ceil{q})=0$ for all $x\in(-\infty, \ceil{q})$.

        To show that $M$ is indeed non-increasing on $(-\infty, \ceil{q}+1)$, write $M$ in Newton form using the derivative form for the divided differences
        \[
        M(x)=h(\ceil{q}+1)+\sum_{r=\ceil{q}+2}^n\frac{h^{(r-\ceil{q}-1)}(\xi_r)}{(r-\ceil{q}-1)!}\prod_{j=\ceil{q}+1}^{r-1}(x-j),
        \]
        for some $\xi_r\in(\ceil{q}+1, r)$. So that
        \[
        M'(x)=\sum_{r=\ceil{q}+2}^n\frac{h^{(r-\ceil{q}-1)}(\xi_r)}{(r-\ceil{q}-1)!}\sum_{i=\ceil{q}+1}^{r-1}\prod_{j=\ceil{q}+1, j\ne i}^{r-1}(x-j),
        \]
        and, in particular, for $r\ge \ceil{q}+2$,
        \begin{align*}
        \sgn\left(\frac{h^{(r-\ceil{q}-1)}(\xi_r)}{(r-\ceil{q}-1)!}\sum_{i=\ceil{q}+1}^{r-1}\prod_{j=\ceil{q}+1, j\ne i}^{r-1}(x-j)\right)&=(-1)^{r-\ceil{q}-1+r-\ceil{q}}\\
        &=-1.
        \end{align*}
        Thereby showing $M'(x)\le 0$ for all $x\in(-\infty, \ceil{q}+1).$

        Finally, to conclude that $L$ is positive in $(\ceil{q},n]$, it suffices to argue that $M$ is positive in $(\ceil{q},n]$. Since $M$ is non-increasing on $(-\infty,\ceil{q}+1]$ and $M(\ceil{q}+1)>0$, it is enough to check the interval $(\ceil{q}+1,n]$.
        For this, observe that on $[\ceil{q}+1,n]$,
        \[
        M(x)=\frac{1}{\pi}\int_0^{\ceil{q}-q}\frac{m_s(x)}{\sqrt{(\ceil{q}-q-s)s}}ds,
        \]
        where $m_s$ is the minimum degree interpolating polynomial for $g_s:x\mapsto \frac{1}{x-\ceil{q}+s}$ at the nodes $\ceil{q}+1, \ldots, n$. The integral identity can be verified by simple computation at the relevant nodes. Since the quantities on both sides are polynomials in $x$, and the degree of the RHS is at most $n-\ceil{q}-1$, we obtain equality by uniqueness of interpolation.

        It is therefore sufficient to show that $m_s(x)\ge 0$ for all $s\in(0,\ceil{q}-q)$ and $x\in[\ceil{q}+1,n]$. It is straightforward to verify that the following is an explicit form for $m_s$ by uniqueness of interpolation
        \[
        m_s(x)=\frac{W(\ceil{q}-s)-W(x)}{W(\ceil{q}-s)(x-\ceil{q}+s)},
        \]
        where $W(x)=\prod_{j=\ceil{q}+1}^n(x-j)$. Moreover, for $k=\ceil{q}+2,\ldots, n$ and $x\in(k-1,k)$, we find
        \begin{align*}
            |W(x)|&=\prod_{j=\ceil{q}+1}^{k-1}|x-j|\prod_{j=k}^{n}|x-j|\\
            &\le \prod_{j=\ceil{q}+1}^{k-1}|k-j|\prod_{j=k}^{n}|j-k+1|\\
            &=(k-\ceil{q}-1)!(n-k+1)!\\
            &\le (n-\ceil{q})!\\
            &\le \prod_{j=\ceil{q}+1}^n|j-\ceil{q}+s|=|W(\ceil{q}-s)|.
        \end{align*}
        In particular,
        \begin{align*}
            m_s(x)=\frac{1}{x-\ceil{q}+s}\left(1-\frac{W(x)}{W(\ceil{q}-s)}\right)\ge0,
        \end{align*}
        for all $x\in[\ceil{q}+1,n]$. This concludes the proof of the lemma.
    \end{proof}

    The next two lemmas show that the polynomial $M$ alternates between over- and under-estimating the target function on successive subintervals and that this behavior carries over to $\sigma_1$. That is, as one moves across the interpolation nodes, $\sigma_1$ lies alternately above and below the curve being interpolated.
    
    \begin{lemma}\label{sign_of_M}
        Let $k=\ceil{q}+1,\ \ldots,\ n$. The function $h-M$ does not change sign on $(k-1,k)$. In particular, for all $x\in(k-1,k)$,
        \[
        \sgn(f(x)-L(x))=(-1)^{k-\ceil{q}-1}=-\sgn\left(V(x)\right).
        \]
    \end{lemma}
    \begin{proof}
        Observe that $h$ is $d$ times continuously differentiable on $(\ceil{q},n]$ with $\sgn(h^{(d)}(\xi))=(-1)^d$ for all $\xi\in(\ceil{q},n]$. By the Lagrange remainder formula, for $k=\ceil{q}+1, \ \ldots, \ n$ and any $x\in(k-1,k)$, we obtain a $\xi\in(\ceil{q},n]$ such that
        \begin{align*}
        \sgn(h(x)-M(x))&=\sgn\left(\frac{h^{(d)}(\xi)}{d!}\prod_{j=\ceil{q}+1}^{n}(x-j)\right)\\
        &=(-1)^d(-1)^{d-\ceil{q}-1+k}\\
        &=(-1)^{k-\ceil{q}-1}.
        \end{align*}
        The second part of the statement follows immediately. 
    \end{proof}
    \begin{lemma}\label{L_relates_to_f}
        Let $k=\ceil{q}+1,\ldots, n$. Then, for all $x\in(k-1,k)$,
        \begin{align*}
        L(x)^2+V(x)>f(x)^2 \text{ or }f(x)^2>L(x)^2+V(x).
        \end{align*}
    \end{lemma}
    \begin{proof}
        Note that $f(x), L(x)>0$ for all $x\in(\ceil{q},n]$ by Lemma \ref{L_on_first_segment}. By Lemma \ref{sign_of_M}, $\sgn(f-L)=-\sgn(V)$. Thus, if $f(x)>L(x)>0$, then $V(x)<0$, so
        \[
        f(x)^2>L(x)^2>L(x)^2+V(x).
        \]
        The other case is similar.
    \end{proof}
    
    We are now ready to show the final lemma. 
    \begin{lemma} \label{roots_of_sigma0_interpolation}
        The only roots of the polynomial $\sigma_0$ in $[0,n]$ are $r\in[\floor{q},\ceil{q})$ and $\ceil{q},\ldots,n$, all of which have odd multiplicities.
    \end{lemma}
    \begin{proof}
        By construction, $\ceil{q},\ldots,n$ are roots. Uniqueness of these roots in the interval $[\ceil{q},n]$ is an immediate consequence of Lemma \ref{L_relates_to_f}. To argue that the roots $x\in[\ceil{q},n]\cap \Z$ have odd multiplicities, just note that
        \[
        \sigma_0(x)=(x-q)(f(x)^2-\sigma_1(x))
        \]
        changes sign at each $x\in[\ceil{q},n]\cap \Z$ by the same lemma.
        Furthermore, there is a unique root~${r\in[\floor{q},q)}$, since
        \begin{align*}
        \frac{1}{\floor{q}-q}\sigma_0(\floor{q})&=\frac{1}{\widehat{q}}-(V(\floor{q})+L(\floor{q})^2)\\
        &=-L(\floor{q})^2\le 0
        \end{align*}
        and
        \[
        \lim_{x\rightarrow q^-}\frac{1}{x-q}\sigma_0(x)=\lim_{x\rightarrow q^-}(f(x)^2-\sigma_1(x))=+\infty,
        \]
        so we may apply the intermediate value theorem combined with $\sigma_1$ being decreasing, since both $L^2$ and $V$ are decreasing on $[\floor{q}, q)$. From this, we also deduce that $r$ is simple, since $f^2-\sigma_1$ is increasing.
        
        To see that there are no roots in $(-\infty,\floor{q})$, it suffices to observe that
        \begin{align*}
        \frac{1}{x-q}\sigma_0(x)=f(x)^2-\sigma_1(x)< f(\floor{q})^2-\sigma_1(\floor{q})\le 0.
        \end{align*}
        For the interval $(q,\ceil{q})$, just note that
        \[
        \frac{1}{x-q}\sigma_0(x)=\frac{x-\ceil{q}}{x-q}-\sigma_1(x)< 0.
        \]
    \end{proof}
    Note that it can likely be shown that the roots are simple, but this is computationally heavier and is not necessary for our purposes.
    
\subsection{Technical Lemmas for Section \ref{lower_layers_ub_section}} \label{technical_section_lower_layers}
In the proof of Theorem \ref{MK-rank}, we left open two Lemmas \ref{choosing constants} and \ref{roots_of_sigma0_Chebyshev}. One about the constraints on~$\sigma_1$ and the other about the roots of $\sigma_0$. We show them here. We will need two helper lemmas taken from \cite{kurpisz2021soscertificationsymmetricquadratic}.
\begin{lemma}[\cite{kurpisz2021soscertificationsymmetricquadratic} Lemma 23] \label{Chebyshev_first_root}
    Let $n,d\in\Z_{\ge 1}$ and let $r_0\in\R$ be the smallest root of $T_d(\frac{x}{n}-1)$. Then $$r_0\le\frac{\pi^2n}{8d^2}.$$
\end{lemma}

\begin{lemma}[\cite{kurpisz2021soscertificationsymmetricquadratic} Lemma 7] \label{Chebyshev_lb}
    Let $n,d\in\Z_{\ge 1}$. Then for all $c\in[0,n]$
    \[
    T_d^2\left(-\frac{c}{n}-1\right)\ge\frac{1}{4}\left(1+\sqrt{\frac{2c}{n}}\right)^{2d}.
    \]
\end{lemma}

    \begin{lemma} \label{choosing constants}
    The polynomial $\sigma_1$
    satisfies the constraints \emph{(1), (2),} and \emph{(3)}.
    \end{lemma}
    \begin{proof}
        Let $N=n-\floor{q}\ge 1$. We show that the three constraints are satisfied separately.
        \begin{enumerate}[label=(\arabic*)]
            \item By standard results on Chebyshev polynomials, we know that $|T_d|$ is decreasing on $(-\infty, x_0]$, where $x_0$ denotes the smallest root of $|T_d|$. It follows that the function $\sigma_1$ is strictly decreasing on $(-\infty,\ceil q]$.        
        It is enough to check that $\sigma_1(\floor{q})=\tau(\floor{q})^2>\frac{1}{\widehat{q}}$. For any $x\in[0,\floor{q}]$, we have:
            \begin{align*}            \sigma_1(x)\ge&\tau(\floor{q})^2\ge\frac{1}{64}\left( \frac{1}{4}\left(1+\sqrt{\frac{2(1-r_0)}{N}}\right)^{d}\right)^{2m}\\
            &\ge \frac{1}{64}\left( \frac{1}{4}\left(1+\frac{1}{\sqrt{N}}\right)^{d}\right)^{2m}\ge \frac{1}{64} \left(\frac{1}{4}2^{d/\sqrt{N}}\right)^{2m}> \frac{1}{\widehat{q}},
            \end{align*}
            where the second inequality follows from Lemma \ref{Chebyshev_lb} applied to $c=1-r_0\in[0,N]$, where ${r_0\le\frac{\pi^2N}{8d^2}\le \frac{1}{2}}$ by Lemma \ref{Chebyshev_first_root} which implies the third inequality. The fourth inequality follows from $(1+\frac{1}{\sqrt{N}})^{\sqrt{N}} \ge 2$ and the last inequality from the definition of $d$ and $m$.
            \item We know that $|T_d(x)|\le 1$ for all $x\in[-1,1]$. 
            Thus, by the Markov brothers' inequality, $|T_d'(x)|\le d^2$ for all $x\in[-1,1]$. 
            In particular, $|T_d'(\frac{x-\ceil{q}+r_0}{N}-1)|\le \frac{d^2}{N}$ for all $x\in[\ceil{q},n]$. Applying the mean value theorem yields
            \begin{align*}               \left|T_d\left(\frac{x-\ceil{q}+r_0}{N}-1\right)\right|&=\left|T_d\left(\frac{x-\ceil{q}+r_0}{N}-1\right)-T_d\left(\frac{r_0}{N}-1\right)\right|\\
                &\le (x-\ceil{q})\sup_{x\in[k,n]}\left|T_d'\left(\frac{x-\ceil{q}+r_0}{N}-1\right)\right|\\
                &\le \frac{d^2}{N}(x-\ceil{q}).
            \end{align*}
            Combining this with $|T_d(\frac{x-\ceil{q}+r_0}{N}-1)|\le 1$ for any $x\in[\ceil{q},n]$ gives
            \begin{align*}
            \sigma_1(x)&=\frac{1}{64} T_d\left(\frac{x-\ceil{q}+r_0}{N}-1\right)^{2m}\\
            &\le \min\left\{\frac{d^{4m}}{64N^{2m}}(x-\ceil{q})^{2m}, \frac{1}{64}\right\}.
            \end{align*}
            For $x\in[\ceil{q},\ceil{q}+1]$, this shows
            \begin{align*}
                \sigma_1(x)&\le \min\left\{\frac{d^{4m}}{64N^{2m}}(x-\ceil{q})^{2m}, \frac{1}{64}\right\}
                \le \frac{x-\ceil{q}}{2},
            \end{align*}
            since $x-\ceil{q}\in[0,1]$ and $\frac{d^{4m}}{64N^{2m}}\le\frac{1}{2}$.
            \item Using the inequality derived in the previous point, we conclude
            \begin{align*}
                \sigma_1(x)\le \min\left\{\frac{d^{4m}}{64N^{2m}}(x-\ceil{q})^{2m}, \frac{1}{64}\right\}\le \frac{1}{64}\le \frac{1}{2}
            \end{align*}
            for any $x\in[\ceil{q}+1,n]$. \qedhere
        \end{enumerate}
    \end{proof}

    \begin{lemma} \label{roots_of_sigma0_Chebyshev}
        The polynomial $\sigma_0$ has exactly two roots in $(\floor{q}, \ceil{q}]$, both of which are simple.
    \end{lemma}
    \begin{proof}
        Firstly, $\sigma_1$ is strictly decreasing on $(q,\ceil q)$ and $\sigma_1(\ceil q)=0$, hence $\sigma_1(x)>0$ for any $x\in(q,\ceil q)$. Thus, 
        \[
        \frac{1}{x-q}\sigma_0(x)=\frac{x-\ceil q}{x-q}-\sigma_1(x)<0
        \]
        on $(q,\ceil q)$, i.e., $\sigma_0$ has no roots in this interval.
        To conclude that there is a root in $(\floor q,q)$, it suffices to observe the following:
        \[
        \frac{1}{\floor q-q}\sigma_0(\floor q)=\frac{1}{\widehat q}-\sigma_1(\floor q)<0
        \]
        and
        \[
        \lim_{x\rightarrow q^-}\frac{1}{x-q}\sigma_0(x)=\lim_{x\rightarrow q^-}\left(\frac{x-\ceil q}{x-q}-\sigma_1(x)\right)=+\infty,
        \]
        and apply the intermediate value theorem. For uniqueness and simplicity, note that $\frac{1}{x-q}\sigma_0(x)$ is increasing on $(\floor q, q)$.
        
        To show that $\ceil{q}$ is simple, note first that the derivative of $\sigma_1$ is given by $\sigma_1'(x)=2\tau'(x)\tau(x)$.
        In particular, $\ceil q$ has multiplicity $1$ since it is a root of both $\sigma_1$ and $\tau$, and
        \[
        \sigma_0'(\ceil q)=1-2(\ceil q-q)\tau'(\ceil q)\tau(\ceil q)-\sigma_1(\ceil q)=1\ne 0.  \qedhere
        \]
    \end{proof}

\bibliographystyle{alpha}
\bibliography{refs}

\end{document}